\documentclass[11pt,a4paper]{article}

\usepackage[utf8]{inputenc}
\usepackage{amsmath,amsthm}
\usepackage{amsfonts}
\usepackage{amssymb}
\usepackage{mathpazo}
\usepackage[pdftex]{graphicx}
\usepackage[all]{xy}

\newcommand{\theory}[1]{\textup{\textsf{#1}}}
\newcommand{\ICL}{\theory{ICL}}
\newcommand{\IPL}{\theory{IPL}}
\newcommand{\CPL}{\theory{CPL}}

\newcommand{\mz}{\textbf{0}}
\newcommand{\mj}{\textbf{1}}
\newcommand{\mzz}{\textbf{0}^2}
\newcommand{\mjj}{\textbf{1}^2}
\newcommand{\mzj}{\textbf{2}}
\newcommand{\vau}{\textbf{3}} 
\newcommand{\es}{\mathcal{S}}

\newcommand{\kor}{\,\mathbf{r}\,}

\newcommand{\oraz}{\,\,\mathrm{and}\,\,}
\newcommand{\lub}{\,\,\mathrm{or}\,\,}
\newcommand{\wtw}{\,\,\mathrm{iff}\,\,}
\newcommand{\pneg}{$\perp$-negation}
\newcommand{\Asim}{{\sim}}

\newtheorem{fct}{Fact}
\newtheorem{thm}[fct]{Theorem}
\newtheorem{prp}[fct]{Proposition}
\newtheorem{cor}[fct]{Corollary}

\theoremstyle{definition}
\newtheorem{dfn}[fct]{Definition}

\newtheorem{example}[fct]{Example}

\title{Negational Fragment \\of Intuitionistic Control Logic}
\author{Anna Glenszczyk}
\date{Draft of \today}

\begin{document}

\maketitle
\begin{abstract}
We investigate properties of monadic purely negational fragment of Intuitionistic Control Logic ($\ICL$). This logic arises from Intuitionistic Propositional Logic ($\IPL$) by extending language of $\IPL$ by additional new constant for {\em falsum}. Having two different falsum constants enables to define two forms of negation. We analyse implicational relations between negational monadic formulae and present a poset of~non equivalent formulae of this fragment of $\ICL$.
\end{abstract}
{\flushleft
{\bf MSC (2010)} Primary: 03B60, 03B20, 03B70; Secondary: 03B62,\\
{\bf Keywords:} Intuitionistic Control Logic, intuitionistic logic, classical logic, Kripke models.}

\section{Introduction}
Intuitionistic Control Logic ($\ICL$) was defined semantically and proof theoretically by Chuck Liang and Dale Miller in their joint work \cite{LMICL}. This logic can be seen as a combination of classical and intuitionistic logics. The original impetus for $\ICL$ came from the search for a~logic that would preserve the crucial connective of intuitionistic implication and at the same time would be able to type programming language control operators such as {\em call/cc}. $\ICL$ adds to the language of $\IPL$ a new constant $\perp$ which is distinct from intuitionistic falsum~$0$. 
Having these two falsum constants $0$ and $\perp$ enables to define two forms of negation: $\Asim  A = A \rightarrow 0$ and $\neg A = A \rightarrow \perp$ respectively.

Let us compare negations in classical and intuitionistic logics. We denote intuitionistic negation by $\Asim  A$ and classical by $\neg A$. It corresponds with notation in $\ICL$ and intuitive meaning of $\neg A$ as "classical" negation in this logic. In Classical Propositional Logic $(\CPL)$ there exist only two non equivalent negational formulae: $A, \neg A$. The classical negation is involutive i.e. $\neg \neg A \leftrightarrow A$ is a $\CPL$ tautology, so it is not possible to define a new operator by iterating classical negation. In $ \IPL $ there are three non equivalent negational formulae: $A, \Asim  A, \Asim  \Asim  A$. It is known that $\Asim  \Asim  A$ does not imply $A$ in intuitionistic logic, but $\Asim  \Asim  \Asim  A \leftrightarrow \Asim  A$ is an intuitionistic tautology. Thus using intuitionistic implication we can obtain a new operation which is a~double negation $\Asim  \Asim $. Further multiplying of intuitionistic negations will only give us (up to equivalence) a formula $\Asim  A$ in case of odd number of negations or a formula $\Asim  \Asim  A$ if there is even number of negations to start with. 

In Intuitionistic Control Logic there are two distinct negations: $\Asim  A$ which is an ordinary intuitionistic negation and $\neg A$ which bears some characteristics of classical negation. Combination of these two kinds of negation results in possibility of forming new operators. We extract them and present their relations with respect to intuitionistic implication. 

\section{Preliminaries}
In this section we recall some facts about $\ICL$ from~\cite{LMICL}. We consider only propositional logic. Language of $\ICL$ consists of countably many variables denoted $p_{1}, p_{2}, p_{3}, \ldots$, intuitionistic connectives $\lor, \land, \rightarrow$ called respectively disjunction, conjunction and implication and of three constants $0, 1, \perp$. As a~shorthand for $(A\rightarrow B) \land (B \rightarrow A)$ we use the expression $A \leftrightarrow B$. A~Kripke model for $\ICL$ will be called an {\em r-model} and is defined as follows.

\begin{dfn}\label{r-model}
A~\textit{Kripke r-model} is a~quadruple of the form $\langle W, \kor,  \leq, \Vdash \rangle$ where $W$ is a finite, non-empty set, $\leq$ is a reflexive and transitive relation on the set $W$ and $\Vdash$ is a binary relation between elements of $W$ and atomic formulae called {\em forcing}. Elements of the set $W$ are called {\em worlds} or {\em nodes}. The~element $\kor \in W$ is the {\em root} of the model. It is the least element of the set $W$ ($\kor \leq u$ for every world $u \in W$).
\end{dfn}
The~forcing relation $\Vdash$ is monotone, that is if $u \leq v$ then $u \Vdash p$ implies $v \Vdash p$. The $\Vdash$ relation is extended to all formulae in the following way. Let $u, v, i \in W$.
\begin{itemize}
\item $u \Vdash 1$ and $u \not \Vdash 0$
\item $\kor \not \Vdash \perp$
\item $i \Vdash \perp$ for all $i > \kor$
\item $u \Vdash A \vee B$ iff $u \Vdash A$ or $u \Vdash B$
\item $u \Vdash A \wedge B$ iff $u \Vdash A$ and $u \Vdash B$
\item $u \Vdash A \rightarrow B$ iff for all $v \geq u$ if $v \Vdash A$ then $v \Vdash B$.
\end{itemize}
If a formula is forced in every world of an r-model, we say that it is satisfied in this r-model. If a formula $A$ is satisfied in all r-models we say that it is valid, in symbols ${}\models A$.

Constants $0$ and $1$ corresponds to intuitionistic falsum and verum. A~formula of $\ICL$ that does not contain constant $\perp$ is an intuitionistic formula. Forcing of $\perp$ distinguishes between the root of an r-model and the rest of worlds. We will call every world properly above the~root an \textit{imaginary world}. We use symbols $u, v , w$ to represent arbitrary worlds in $W$ and the symbol $i$ to represent an imaginary world.

Because of two different constants for falsum, it is possible to define two different negations \\
\begin{tabular}{lr}
\rule[-2ex]{0pt}{5.5ex}
Intuitionistic negation: & $\Asim  p \,= \,p \rightarrow 0$ \\ 
\rule[-2ex]{0pt}{5.5ex} 
Classical negation: & $\neg\, p \,= \,p \rightarrow \perp$ \\ 
\end{tabular} 

The~term \textit{classical} in the~name of the~second negation refers to the~law of excluded middle which, with respect to this negation, is an~$\ICL$ tautology. Let us suppose that $p \lor \neg p$ is refuted in the~root of some r-model: $$ \kor \not \Vdash p \lor \neg p.$$ This is equivalent to $$\kor \not \Vdash p \oraz \kor \not \Vdash \neg p$$ which implies that $\kor \not \Vdash p$ and there exists a world $u \geq \kor$ such that  $u \Vdash p$ and $u \not \Vdash \perp$. Condition $u \not \Vdash \perp$ means that $u = \kor$. Hence we get a~contradiction. Nevertheless this negation is not fully classical --- it is not involutive as $\neg \neg p $ does not imply $p$. It is because this negation is defined using intuitionistic implication. For this reason we prefer to call it $\perp$-{\em negation}, instead of "classical" negation.

In \cite{LMICL} Liang and Miller defined sequent calculus LJC for Intuitionistic Control Logic and proved soundness and completeness of LJC with respect to the Kripke semantics. However, in this paper we focus on the semantical approach and the equivalence between provability of a formula and its validity in all r-models is alluded to only in Theorem \ref{extensionality}. The symbol $\vdash A$ denotes provability of a formula $A$ in LJC. 

For a background in intuitionistic logic see \cite{MIN}.

\section{Negational fragment}
In this paper we will consider monadic purely negational fragment of $\ICL$, i.e. the~fragment in the~language of $\Asim $, $\neg$ and $p$ only. It means that we treat both negations as primitive connectives, not defined by means of constants and implication. Formulae of this fragment will be called \textit{n-formulae}. 

It~will be understood that $N_{k}$ and $N_{m}$ are different sequences of both negations and that $k, m \in \{0, 1, 2, \ldots\} = \mathbb{N}$.  To discriminate sequences of the same length we will use superscripts $N_{k}^{1}, N_{k}^{2}$ etc. By $\Asim^n$ and $\neg^n$ we will understand iteration of $n$ negations of given kind. 
We will denote by $N_{k} p$ an~n-formula with $k$ negations of both kinds. By \textit{the~length of an n-formula} $N_k p$ we define the number $k$ of negations. Formulae of the~form $N_{2j} p$ and $N_{2j+1} p$ will be called {\em even n-formula} and {\em odd n-formula}, respectively. We will treat the variable $p$ as a negational formula of the length $0$.

Every r-model defined as in Definition~\ref{r-model} is a model for monadic purely negational fragment of $\ICL$ as well. However, since negations $\Asim$ and $\neg$ are our primitive notions, considering r-models for negational fragment we should define interpretation of these connectives independently.
\begin{dfn}\label{for:fragm}
A~Kripke model for the negational fragment of $\ICL$ is a tuple $\mathcal{M} = \langle W, \kor,  \leq, \Vdash \rangle$ where $W$, $r$ and $\leq$ are defined as in Definition~\ref{r-model} and the forcing relation $\Vdash$ is restricted to the variable $p$ and constants $0$, $\bot$. Additionally we define the interpretation of negations:
\begin{itemize}
\item $u \Vdash \Asim  A \wtw w \not \Vdash A, \,\mathrm{for\, all}\, w \geq u$
\item $u \Vdash \neg A \wtw  w \not \Vdash A \lub w > \kor, \,\mathrm{for\,all}\, w \geq u$.
\end{itemize}
\end{dfn}
It is according to the definition of forcing for constants $0, \perp$ and intuitionistic implication in the case of full language. Forcing of intuitionistic negation is standard. For \pneg\, we have \[ u \Vdash \neg p \wtw  w \not \Vdash p \lub w \Vdash \perp,\,\mathrm{for\, all}\, w \geq u.\] The condition $w \Vdash \perp$ means that $w$ is an imaginary world. 
\begin{fct}
\label{f:kor}
For every n-formula $A$ we have:
\begin{enumerate}
\item\label{kor:f} $\kor \Vdash \neg A \wtw \kor \not \Vdash A$,
\item\label{kor:r} $\kor \not \Vdash \neg A \wtw \kor \Vdash A$,
\item\label{pneg:r} $u \not \Vdash \neg A \wtw u = \kor \oraz u \Vdash A \, \mathrm{for\, arbitrary\,} \,\, u \in W$,
\item $\,i \,\Vdash \neg A \, \mathrm{for\, all\,} \,\, i > \kor$.
\end{enumerate}
\end{fct}
The first point is straightforward from definition. In $\ICL$ the distinction between the root of the r-model and other worlds is expressed by the forcing of $\perp$, whereas in the monadic purely negational fragment the root of the model is the only world in which \pneg\, of a~formula can be refuted: \[ u \not \Vdash \neg A \wtw w \Vdash A \oraz w = \kor, \,\mathrm{for\, some\,} w \geq u. \] It follows that \pneg\, of a formula is forced in every imaginary world.

We are interested in relations between n-formulae and we investigate validity of formulae of the form $N_{k} p \rightarrow N_{m} p$. We denote by $\mathcal{N}$ the set of all n-formulae. In the standard way we define an~equivalence relation $\equiv$ on the set $\mathcal{N}$: $$A \equiv B \, \mathrm{iff} \, \models A \rightarrow B \oraz \models B \rightarrow A.$$ As usual, we consider the quotient set:
$$\mathcal{N}/_{\equiv} = \{[A]_{\equiv} \,\, | \,\, A \in \mathcal{N} \},$$
where $[A]_\equiv$ is the equivalence class of a formula $A$. The relation $\preceq$ on $\mathcal{N}/_{\equiv}$ is given by:
$$[A]_\equiv \preceq [B]_\equiv \,\, \mathrm{iff} \,\, \models A \rightarrow B.$$ Although the relation $\preceq$ is defined on equivalence classes, no confusion should arise if we use it to denote a relation between two n-formulae:
$$A \preceq B \,\, \mathrm{iff}\,\, \models A \rightarrow B.$$ If for an n-formula $A = N_k p$ exists an n-formula $B = N_m p$ such that $A \equiv B$ and $m < k$, we say that $A$ {\em is reducible to} $B$. In the other case we say that an n-formula $A$ is {\em irreducible}. 

Most proofs of facts about implicational relations between n-formulae are reduced to showing a contradiction in the procedure of finding a countermodel for a formula $A \rightarrow B$. For n-formulae of a length greater that $4$ we repeatedly refer to extensionality. 

\begin{thm}\label{extensionality}
For any formula $A(\bar{p}, s)$ and for all formulae $B, C$ if $\vdash B \leftrightarrow C$ then $$\vdash A(\bar{p}, B/s) \leftrightarrow A(\bar{p}, C/s).$$
\end{thm}

It is well-known that this theorem holds for $\IPL$. The proof is by induction on the complexity of~formulae. However, the case of additional constant $\perp$ does not interfere with the proof, thus the theorem holds for $\ICL$ as well. 

\section{Relations between models}
In a Kripke model for either $\ICL$ or its negational fragment let $\circ$ mark a~node in which a variable $p$ is refuted and $\bullet$ a node in which $p$ is forced.

Firstly, let us consider two basic n-formulae $\Asim  p$ and $\neg p$. It is easy to see that minimal model and countermodel for $\Asim  p$ are  $\circ$ and $\bullet$ respectively. In case of $\neg p$ due to Fact~\ref{f:kor} we have $\kor \Vdash \neg p \wtw \kor \not \Vdash p $ and $\kor \not \Vdash \neg p \wtw \kor \Vdash p$. Thus a minimal model for $\neg p$ is $\circ$ and minimal countermodel is $\bullet$. 

The fact that $\Asim  p$ and $\neg p$ have the same minimal models and countermodels does not imply that these formulae are equivalent. There is a model in which $\neg p$ is satisfied and $\Asim  p$ is refuted, namely:
$$
  \xymatrix{\bullet\ar@{-}[d]\hspace{1pt} i \\ \circ \hspace{1pt}\kor}
$$

It is easy to see, that for every n-formula $N_{k} p$ the minimal model or countermodel are $\circ$ and $\bullet$. These cases may seem not very interesting as they collapse both negations to the situation of ordinary classical negation. However, looking for a countermodel for an intuitionistic implication of two formulae is equivalent to looking for a model for the antecedent and a countermodel for the consequent. While considering an~implication of n-formulae, one of these cases are frequently reduced to either $\circ$ or $\bullet$, so it is sufficient to know if the variable can be forced or refuted in a given world of the model. This depends on the evenness of the sequence of negations preceding the variable. The following fact becomes useful in such situations.

\begin{prp}
\label{forcing}
For every world $u$ in an r-model $\mathcal{M}$ we have:
\begin{enumerate}
\item\label{f:even} if $u \Vdash N_{2k} p$ then $w \Vdash p \lub w > \kor$, for some $w \geq u$,
\item\label{f:odd} if $u \Vdash N_{2k+1} p$ then $w \not \Vdash p \lub w > \kor$, for some $w \geq u$,
\item\label{r:even} if $u \not \Vdash N_{2k} p$ then $w \not \Vdash p \lub w > \kor$, for some $w \geq u$,
\item\label{r:odd} if $u \not \Vdash N_{2k+1} p$ then $w \Vdash p \lub w > \kor$, for some $w \geq u$.
\end{enumerate} 
\end{prp}

\begin{proof}
We prove only one of the two most complex cases which is~\ref{f:odd}, others can be proven in an analogous way.

Let $\mathcal{M}$ be an r-model and $u$ an arbitrary world in this model. The proof is by induction on $k$.

Let $k=0$ and let $A= N_1 p$. We have to consider two cases:

Case 1. $A= \Asim p$.\\
Assume that $u \Vdash \Asim p$. Then for all $w \geq u$ we have $w \not \Vdash p$ and the claim trivially follows.

Case 2. $A = \neg p$.\\
From the assumption that $u \Vdash \neg p$ it follows that either $u = \kor$ and $u \not \Vdash p$ or $u > \kor$. In both cases the claim follows.

For the induction step, let $k>0$ and let $A=N_{2(k+1)+1} p$. Now, we consider following cases: 

Case 1. $A= \Asim \Asim N_{2k+1} p$.\\
Assume that $$u \Vdash \Asim \Asim N_{2k+1} p.$$ Then for every  $v \geq u$ there exists a world $v' \geq v$ such that $v' \Vdash N_{2k+1} p$. Hence, by the induction hypothesis, there is $w \geq v'$ with $$w \not \Vdash p \lub w > \kor.$$
In particular it follows that the claim holds.

Case 2. $A= \Asim \neg N_{2k+1} p$.\\
From the assumption  $$u \Vdash \Asim \neg N_{2k+1} p$$ we get that for every world $u' \geq u$ we have $u' \not \Vdash \neg N_{2k+1} p$. From point~\ref{pneg:r} of Fact~\ref{f:kor} it follows that $\kor$ is the only world of the model and $$\kor \Vdash N_{2k+1} p,$$ which, by the induction hypothesis, obviously implies the claim.

Case 3. $A= \neg \Asim N_{2k+1} p$.\\
Then from $$u \Vdash \neg \Asim N_{2k+1} p$$ it follows that, in particular, $$u \not \Vdash \Asim N_{2k+1} p \lub u > \kor.$$ If the latter holds, we are done. The former implies that there is $u' \geq u$ with $$u' \Vdash N_{2k+1} p$$ and the claim follows from the induction hypothesis.

Case 4. $A= \neg \neg N_{2k+1} p$.\\
Assume that $$u \Vdash \neg \neg N_{2k+1} p.$$ Then it follows that either $u > \kor$, in which case the claim holds, or $u= \kor$ and $u \not \Vdash \neg N_{2k+1} p$. If so, from Fact~\ref{f:kor} we have $\kor \Vdash N_{2k+1} p$ and again, by the~induction hypothesis, we get the claim.
\end{proof} 

Let us consider an implication of two different n-formulae $N_{k} p \rightarrow N_{m} p$. Such a formula is never valid if the evenness  of $k$ and $m$ is not the same. Indeed, suppose that
\[ \kor \not \Vdash N_{k} p \rightarrow N_{m} p \] and let $k = 2j_1, m = 2j_2+1$ for some $j_1,j_2 \in \mathbb{N} $. Then there exists a world $u \geq \kor$ such that
\[ u \Vdash N_{2j_1} p \oraz u \not \Vdash N_{2j_2+1} p. \]
The countermodel for such a formula is $\bullet$. The other case of evenness is symmetrical with a countermodel $\circ$.

In fact for an implication of two different n-formulae we never need a countermodel of height greater than $2$. That is because the implication which bounds two n-formulae is the only connective that in building a~countermodel requires creating a new world possibly above the root. The search for a countermodel for a formula $$A = N_{k} p \rightarrow N_{m} p$$ always starts with the assumption that $\kor \not \Vdash A$ which is equivalent to the~fact that there is a world $u \geq \kor$ such that $$u \Vdash N_{k} p \oraz u \not \Vdash N_{m} p.$$ 

The minimal countermodel for an n-formula is either $\circ$ or $\bullet$. Intuitionistic negation influences only the forcing of a subformula at the given world, regardless what kind of the world it is. Considering the~\pneg\, of a~formula requires discriminating the root from imaginary worlds. In other worlds, forcing of the variable in a given world is one of the two ways of distinguishing worlds in a model. The other one, as was already said, is refuting the \pneg \, of a formula. 
\begin{example}\label{ex:jed}
Let $A = \neg \neg \Asim  p$ and $B = \neg \Asim  \Asim  p$. We will show that $A \preceq B$ and $B \not \preceq A$.

Suppose that there exists an r-model in which $A \rightarrow B$ is not satisfied i.e. this formula is refuted in some world of this model. Showing a contradiction will give us $A \preceq B$. 

Refutation of a formula in some world of an r-model means that it cannot be forced in the root of the model: $\kor \not \Vdash A \rightarrow B$, so there exists a~world $u$ possibly above the root, such that
$u \Vdash \neg \neg \Asim  p \oraz u \not \Vdash \neg \Asim  \Asim  p.$
The~first condition says that either the world $u$ is imaginary or $u = \kor $ and the~variable $p$ is not forced in any node of the model. The latter condition implies that $u = \kor$ and for every world $w$ in the model exists $w' \geq w$ with $w' \Vdash p$, hence a~contradiction. It follows that $A \preceq B$.

Suppose that for some r-model $\kor \not \Vdash B \rightarrow A$. Then again there exists a~world $u$ possibly above the root in which $u \Vdash \neg \Asim  \Asim p \oraz u \not \Vdash \neg \neg \Asim  p.$ Refuting the~formula $\neg \neg \Asim p$ in some world of a model means that this world is the root and there must exists a node somewhere in the model in which $p$ is forced. If so, forcing of the formula $\neg \Asim \Asim p$ implies that there exists a~world $w \geq \kor$ for which $w' \not \Vdash p$ for every $w' \geq w$, so in particular $\kor \not \Vdash p$. The least possible countermodel is:
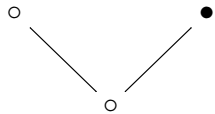
\[ \xymatrix{\circ\ar@{-}[dr] & {} & \bullet\ar@{-}[dl]\\ &\circ& } \]
The first part of the example shows that in looking for a countermodel the~{\em kind} of a current world (either the root or imaginary) is important. The~second part shows the difference of forcing of the variable in imaginary worlds. No bigger models would be necessary, as there is only one variable to validate and only one intuitionistic implication. 
\end{example}
In other worlds forcing or refuting an intuitionistic negation of a~formula in a world possibly above the root cannot extort creating a new world properly above. Forcing of \pneg\, of a formula in the root depends on the forcing of the variable, above the root it is always forced. The case of refuted \pneg\, of a formula sends us back to the root.

For a monadic negational formula there are two possible countermodels of the~height of $1$ and four possible countermodels of the height of $2$:

$$
  \xymatrix{\\ \circ} \quad
  \xymatrix{\\ \bullet}\quad 
  \xymatrix{\circ\ar@{-}[d]\\ \circ}\quad
  \xymatrix{\bullet\ar@{-}[d]\\ \circ}\quad
  \xymatrix{\bullet\ar@{-}[d]\\ \bullet}\quad
  \xymatrix{\circ\ar@{-}[dr] & {} & \bullet\ar@{-}[dl]\\ &\circ& }
$$
$$
  \xymatrix{\mz}\quad  
  \xymatrix{\mj}\quad 
  \xymatrix{\mzz}\quad
  \xymatrix{\mzj}\quad
  \xymatrix{\mjj}\quad
  \xymatrix{&\vau&}
$$
Let us denote the set $\{\mz, \mj, \mzz, \mzj, \mjj, \vau\}$ of these models by $\es$.

We are interested in finding the upper bound of the number of non equivalent n-formulae. We are looking for n-formulae $ A $ and $ B $ such that $ A \not \preceq B $. It means that there exists a model for the formula $A$ in which we can refute the~other n-formula.  To every implication $A \rightarrow B$ we can assign a~subset of $\es$ of models in which this formula is refuted. Such subset cannot contain both of the~models $\circ$ and $\bullet$, because contradiction is not expressible in our language. There are $2^{5}$ such subsets of $\es$, so there are at most $32$ non equivalent monadic n-formulae. 

For a given n-formula every model from the set $\es$ can be either a model or a~countermodel. The existence of only six possible models for monadic n-formulae enables to characterize these formulae in terms of their models. Let '+' stand for 'valid', and '-' for 'not valid'.  For example for the formula $p$ we have
\begin{center}
\begin{tabular}{|c|c|c|c|c|c|c|}
\hline \rule[-2ex]{0pt}{5.5ex}
$N_0$ & $\mz$ & $\mj$ & $\mzz$ & $\mzj$ & $\mjj$ & $\vau$ \\ 
\hline\rule[-2ex]{0pt}{5.5ex}  $p$ 
& - & + & - & - & + & - 
\\ 
\hline 

\end{tabular} 
\end{center}
and for n-formulae of the length $2$ we have
\begin{center}
\begin{tabular}{|c|c|c|c|c|c|c|}
\hline\rule[-2ex]{0pt}{5.5ex} 
$N_2$ & $\mz$ & $\mj$ & $\mzz$ & $\mzj$ & $\mjj$ & $\vau$ \\ 
\hline\rule[-2ex]{0pt}{5.5ex}  
$\Asim \Asim  p$ 
& - & + & - & + & + & -
\\ 
\hline\rule[-2ex]{0pt}{5.5ex}  
$\Asim \neg p$ 
&  - & +  & -  & - & - & -  
\\ 
\hline\rule[-2ex]{0pt}{5.5ex}  
$\neg\Asim  p$ 
& - & + & - & + & + & +
\\ 
\hline\rule[-2ex]{0pt}{5.5ex}  
$\neg\neg p$ 
& - & +  & -  & - & +  & - 
\\ 
\hline 
\end{tabular} 
\end{center}
It can be seen that there are two countermodels for a formula $\Asim  \Asim  p \rightarrow \Asim  \neg p$, namely $\mzj$ and $\mjj$. However, we could also see that sets of models for $p$ and $\neg \neg p$ are the same, and it is known, that \pneg\, is not involutive. Obviously some informations are missing.

In the case of $\IPL$, when we look for a countermodel of a formula $A\rightarrow B$, we can always start with the assumption that it is already falsified in the root of the model, i.e. we can assume that $\kor \Vdash A \oraz \kor \not \Vdash B$. That is because in $\IPL$ the root of the model is no different from other worlds. We already saw that in $\ICL$ there is a considerate difference between the root and any imaginary world, e.g. the root is the only world of the model in which $\neg A$ can be refuted. It is not enough to look for countermodels of an n-formula starting in the root of an r-model. The forcing of it has to be also revised in those possible imaginary worlds of models $\mzz, \mzj, \mjj \oraz \vau$. It is not sufficient to examine sets of models and countermodels for each n-formula. Its validity has to be also verified in every \textit{pseudosubmodel}. Here by the pseudosubmodel we mean any generated submodel in the sense of $\IPL$ which is not an r-model, that is which consists imaginary worlds only. Let us denote $i(\mz), i(\mj)$ and $i(\mz , \mj)$ pseudosubmodels of $\mzz, \mjj, \vau$ respectively. The pseudosubmodel of $\mzj$ is the same as that of $\mjj$. 

For $p$ and $\neg \neg p$ we have
\begin{center}
\begin{tabular}{|c|c|c|c|}
\hline \rule[-2ex]{0pt}{5.5ex}
 & $i(\mz)$ & $i(\mj)$ & $i(\mz , \mj)$ \\ 
\hline\rule[-2ex]{0pt}{5.5ex}  $p$ 
& - & + & + \\ 
\hline\rule[-2ex]{0pt}{5.5ex}  
$\neg\neg p$ 
& + & + & + \\ 
\hline 
\end{tabular} 
\end{center}
Indeed, assume that \[ \kor \not \Vdash \neg \neg p \rightarrow p.\] Then in some world $u \geq \kor$ we have $u \Vdash \neg \neg p \oraz u \not \Vdash p.$ Forcing of a double \pneg \, of the variable in a world possibly above the root implies that $$\mathrm{either}\, u >\kor \lub u = \kor \oraz u \Vdash p,$$ 
hence the countermodel for a formula $\neg \neg p \rightarrow p$ is $\mzz$.

\section{Semantic characterisation of n-formulae}
It was already emphasised that no n-formula with odd number of negations can be equivalent to n-formula with even number of negations. Therefore all properties concerning equivalences between n-formulae are divided into two cases: for odd and even length of n-formulae. The property of reduction of negations with respect to sequences of one type of negation is straightforward and it follows from the fact that both $\Asim  \Asim  \Asim  p \equiv \Asim  p$ and $\neg \neg \neg p \equiv \neg p$. 
\begin{prp}\label{eq:skr}
For any $k \in \mathbb{N}$ we have:
\begin{enumerate}
\item\label{sim:even} $\Asim ^{(2k+2)} p \equiv \Asim  \Asim  p$,
\item\label{sim:odd} $\Asim ^{(2k+1)} p \equiv \Asim  p$,
\item\label{neg:even} $\neg^{(2k+2)} p \equiv \neg \neg p$,
\item\label{neg:odd} $\neg^{(2k+1)} p \equiv \neg p$.
\end{enumerate}
\end{prp}
\begin{proof}
Ad \ref{sim:even} and \ref{sim:odd}\\
Equivalence between $\Asim  p$ and $\Asim  \Asim  \Asim  p$ is an intuitionistic tautology. From this and extensionality we get the thesis.

Ad \ref{neg:even} and \ref{neg:odd}\\
Let us see that $\neg \neg \neg p \equiv \neg p$.
Suppose that $\not \models \neg \neg \neg p \rightarrow \neg p$, so there exists an~r-model in which $\kor \not \Vdash \neg \neg \neg p \rightarrow \neg p$. Then there exists a world $u \geq \kor$ such that
\[ u \Vdash \neg \neg \neg p \oraz u \not \Vdash \neg p.\] According to point~\ref{pneg:r} of Fact~\ref{f:kor}, refuting the  \pneg \, of a formula in some world of a model sends us back to the root, so the latter condition implies that $u= \kor$ and $u \Vdash p$. If so, from the~former condition and point \ref{kor:f} of Fact \ref{f:kor} it follows that in the root the~formula $\neg \neg p$ is refuted. A~contradiction, since this means that $\kor \not \Vdash p$.

Suppose that the formula $\neg p \rightarrow \neg \neg \neg p$ is not valid. Then in the root of some r-model we have $\kor \not \Vdash \neg p \rightarrow \neg \neg \neg p$. Hence for a world $u$ possibly above the root we have
\[ u \Vdash \neg p \oraz u \not \Vdash \neg \neg \neg p. \]
Similar argumentation as in previous case with respect to the second condition shows that $u = \kor$ and $u \Vdash \neg \neg p$. From points \ref{kor:f} and \ref{kor:r} of Fact \ref{f:kor} we have that $\kor \Vdash p$. On the other hand, since $u = \kor$, from $u \Vdash \neg p$ follows $\kor \not \Vdash p$, a~contradiction. Therefore $\neg \neg \neg p \equiv \neg p$ and the claim follows from extensionality.
\end{proof}

In \cite{LMICL} Liang and Miller distinguished a formula $\Asim  \neg A \rightarrow A$. It enables to emulate the $\mathcal{C}$ control operator. From point \ref{simneg2} of the following proposition it follows that the formula $\Asim  \neg p$ is a representative of a wide class of equivalent n-formulae of the~form $\Asim  \neg N_{2k} p$. Point \ref{simneg3} shows a similar result for a class of odd n-formulae.

\begin{prp}\label{simneg:skr}
For any $k \in \mathbb{N}$ we have:
\begin{enumerate}
\item\label{simneg2} $\Asim  \neg N_{2k} p \equiv \Asim  \neg p$,
\item\label{simneg3} $\Asim  \neg N_{2k+1} p \equiv \Asim  \neg \neg p$.
\end{enumerate}
\end{prp}
\begin{proof}
Ad \ref{simneg2}\\
Assume there exists an r-model $\mathcal{M}$ in which $\kor \not \Vdash \Asim  \neg N_{2k} p \rightarrow \Asim  \neg p$. Thus there exists a world $u \geq \kor$ for which we have
\[ u \Vdash \Asim  \neg N_{2k} p \oraz u \not \Vdash \Asim  \neg p, \] that is in all worlds above the world $u$ the formula $\neg N_{2k} p$ is refuted, which means that the possible countermodel consists of the root only and that $\kor \Vdash N_{2k} p$. Thus from point~\ref{f:even} of Proposition~\ref{forcing} we have $\kor \Vdash p$. On the other hand, refuting the formula $\Asim \neg p$ in the root in particular implies that $\kor \not \Vdash p$. A contradiction.

Let $\kor$ be the root of some r-model in which a formula $\Asim  \neg p \rightarrow \Asim  \neg N_{2k} p$ is not satisfied. Then for some world $u \geq \kor$ we claim that
\[u \Vdash \Asim  \neg p \oraz u \not \Vdash \Asim  \neg N_{2k} p.\] As could be already seen in the case of reverse implication, forcing of the~formula $\Asim \neg p$ in arbitrary world implies that the model is reduced to the root and $\kor \Vdash p$. From point~\ref{r:even} of Proposition~\ref{forcing} with respect to the second condition and the fact that the model comprises the root only, it yields that $\kor \not \Vdash p$. From which we get a~contradiction and as a result the claim follows.

Ad \ref{simneg3}\\
Suppose that $\not \models \Asim  \neg N_{2k+1} p \rightarrow \Asim  \neg \neg p$. So there exists an r-model such that for the root of it we have $$\kor \not \Vdash \Asim  \neg N_{2k+1} p \rightarrow \Asim  \neg \neg p$$ which implies that there exists a world $u \geq \kor$ such that
\[u \Vdash \Asim  \neg N_{2k+1} p \oraz u \not \Vdash \Asim  \neg \neg p. \] Similar argumentation as in the proof of point~\ref{simneg2} and application of point~\ref{f:odd} of Proposition~\ref{forcing} to the first condition implies that for all worlds $u' \geq u$ we have $$u' = \kor \oraz u' \not \Vdash p.$$ If the root is the only world of the model, then in particular $u \not \Vdash \Asim \neg \neg p$ implies $u = \kor \Vdash p$, a contradiction. 

Let us consider an r-model in which $\Asim  \neg \neg p \rightarrow \Asim  \neg N_{2k+1} p$ is not satisfied. Then this formula is refuted in the root of this model and there exists a world $u \geq \kor$ such that
\[ u \Vdash \Asim  \neg \neg p \oraz u \not \Vdash \Asim  \neg N_{2k+1} p. \] The first condition implies that in every world $u' \geq u$ the formula $\neg \neg p$ is refuted and that means $$u' = \kor \oraz u' \not \Vdash p,$$ whereas from the second condition follows that there is a world $v \geq u$ such that $v \Vdash \neg N_{2k+1} p$.
If so, by application of the point~\ref{r:odd} of Proposition~\ref{forcing} and the fact of the root being the only world of the model, we have in particular $$v' = \kor \oraz v' \Vdash p,$$ a contradiction. Thus the thesis holds.
\end{proof}
Proposition~\ref{eq:skr} and Proposition~\ref{simneg:skr} show that in many cases we can reduce an n-formula of a~greater length to a formula of length less than $4$.

In fact the formula $\Asim  \neg p$ implies every n-formula with even number of negations. That is because forcing of the formula $\Asim \neg A$ at a given world sends us back to the root of the model in which $A$ must be forced. Thus we get a~minimal element with respect to relation $\preceq$ for the subset of even n-formulae. Analogically the formula $\Asim  \neg \neg p$ is the minimal element for the~subset of odd n-formulae. These two facts are corollaries from points \ref{simneg2} and \ref{simneg3} of Proposition~\ref{simneg:skr} and the following proposition: 
\begin{prp}\label{lem:MIN}
For every $k, m \in \mathbb{N}$ following implications hold:
\begin{enumerate}
\item\label{sn:even} $\Asim  \neg N_{2k} p \rightarrow N_{2m} p$,
\item\label{sn:odd} $\Asim  \neg N_{2k+1} p \rightarrow N_{2m+1} p$.
\end{enumerate}
\end{prp}
\begin{proof}
Ad \ref{sn:even}\\
Assume $\kor \not \Vdash \Asim  \neg N_{2k} p \rightarrow N_{2m} p$ for the root of some r-model $\mathcal{M}$. Then in some world $u \geq \kor$ we have $u \Vdash \Asim  \neg N_{2k} p \oraz u \not \Vdash N_{2m} p$. From the~former it follows that the root is the only world of the model $\mathcal{M}$ and $\kor \not \Vdash \neg N_{2k} p$. According to point~\ref{kor:r} of Fact~\ref{f:kor} it follows that $\kor \Vdash N_{2k} p$. On~the~other hand $\kor \not \Vdash N_{2m} p$, a~contradiction. 

Ad \ref{sn:odd}\\
Let $\mathcal{M}$ be an r-model in which $\kor \not \Vdash \Asim  \neg N_{2k+1} p \rightarrow N_{2m+1} p$. Then there exists a world $u \geq \kor$ such that $u \Vdash \Asim  \neg N_{2k+1} p \oraz u \not \Vdash N_{2m+1} p$. Hence, due to a similar reasoning as in point~\ref{sn:even}, since the root is the only world of the model we have $\kor \Vdash N_{2k+1} p \oraz \kor \not \Vdash N_{2m+1} p$, a contradiction. 
\end{proof}
It is worth noting that n-formulae $\Asim  \Asim  \neg \neg p$ and $\Asim  \Asim  \neg p$ are maximal elements with respect to the relation $\preceq$ for subsets of even n-formulae and odd n-formulae, respectively.
\begin{prp}\label{lem:MAX}
For any $k \in \mathbb{N}$ following implications are valid:
\begin{enumerate}
\item\label{LA:even} $N_{2k} p \rightarrow \Asim  \Asim  \neg \neg p$,
\item\label{LA:odd} $N_{2k+1} p \rightarrow \Asim  \Asim  \neg p$.
\end{enumerate}
\end{prp}
\begin{proof}
Ad \ref{LA:even}\\
Let $\mathcal{M}$ be an r-model in which $\kor \not \Vdash N_{2k} p \rightarrow \Asim  \Asim  \neg \neg p.$ Then there exist a~world $u$ possibly above the root such that $u \Vdash N_{2k} p \oraz u \not \Vdash \Asim  \Asim  \neg \neg p$. The latter implies that the model $\mathcal{M}$ consists of only one element, namely the root and $\kor \not \Vdash p$. If there is no worlds properly above the root, then from $u \Vdash N_{2k}p$, according to point~\ref{f:even} of Proposition~\ref{forcing} it follows in particular that $\kor \Vdash p$, a~contradiction.

Ad \ref{LA:odd}\\
Suppose that there exists an r-model $\mathcal{M}$ in which $\kor \not \Vdash N_{2k+1} p \rightarrow \Asim  \Asim  \neg p$. Then there exists a world $u \geq \kor $ such that $u \Vdash N_{2k+1} p \oraz u \not \Vdash \Asim  \Asim  \neg p$. 
According to point~\ref{f:odd}~of~Proposition~\ref{forcing} it follows that there exists a world $w \geq u$ such that $w \not \Vdash p \lub w > \kor$. On the other hand, refuting the formula $\Asim \Asim \neg p$ in an arbitrary world of the model means that there are no worlds properly above the root and $\kor \Vdash p$. We have a contradiction and hence the~claim holds.
\end{proof}

For every $k \in \mathbb{N}$ there are $2^{k}$ n-formulae $N_{k} p$. It was already said that there are at most $32$ non equivalent n-formulae. Procedure of finding these n-formulae is reduced to checking if the relation $N_{k} p \preceq N_{m} p$ holds. It would be arduous if it weren't for the fact that we can characterize a~negational formula in terms of its models and countermodels. Instead of checking satisfiability of formulae of the form $N_{k} p \rightarrow N_{m} p$ for subsequent n-formulae $N_k p, N_m p$, it is sufficient to compare sets of models and countermodels, including pseudosubmodels, for these n-formulae. For a given n-formula $N_{k}p$ let $\es^{+}(N_{k} p)$ be the subset of $\es \cup \{i(\mz), i(\mj), i(\mz, \mj)\}$ of models in which n-formula $N_{k} p$ is valid. The relation $N_{k} p \preceq N_{m} p$ between two n-formulae holds if and only if $\es^{+}(N_{k} p) \subseteq \es^{+}(N_{m} p)$. Complete tables of models for a given n-formula up to the length $5$ are given in the Appendix.

Semantically all proofs of following facts are similar to the proof of Proposition~\ref{eq:skr}. They are not informative, thus omitted. 

As a representative for every equivalence class we choose a formula of the smallest length. We start with three equivalence classes of the simplest n-formulae, namely $$[p]_\equiv, [\Asim p]_\equiv \oraz [\neg p]_\equiv$$ and we will shortly discuss subsequent n-formulae.
\begin{fct}\label{N2:irr}
All n-formulae of length $2$ are pairwise non equivalent. 
\end{fct}
As none of n-formulae $N_2 p$ is reducible, we can distinguish four different equivalence classes: $$[\Asim \Asim p]_\equiv, [\Asim \neg p]_\equiv, [\neg \Asim p]_\equiv, [\neg \neg p]_\equiv.$$
There are $8$ n-formulae $N_3 p$. From Proposition~\ref{eq:skr} if follows that $\Asim  \Asim  \Asim  p$ and $\neg \neg \neg p$ are reducible to $\Asim  p$ and $\neg p$ respectively. 
\begin{fct}\label{N3:irr}
There are only two irreducible formulae $N_3 p$ that are not equivalent to any other n-formulae of the length 3, namely $\neg \Asim  \Asim  p \oraz \neg \neg \Asim  p.$ 
\end{fct}

\begin{fct}\label{eq:3}
For n-formulae of the length $3$ we have following equivalences:
\begin{enumerate}
\item\label{eq:jed} $\Asim  \Asim  \neg p \equiv \neg \Asim  \neg p$,
\item\label{eq:dw} $\Asim  \neg \Asim  p \equiv \Asim  \neg \neg p$.
\end{enumerate} 
\end{fct}

From Proposition~\ref{eq:skr}.\ref{sim:odd}, Proposition~\ref{eq:skr}.\ref{neg:odd}, Fact~\ref{N3:irr} and Fact~\ref{eq:3} it follows that there are only 4 irreducible and not equivalent n-formulae $N_3 p$. We choose following representatives: $$[\Asim  \neg \neg p]_\equiv, [\neg \neg \Asim  p]_\equiv, [\neg \Asim  \Asim  p]_\equiv \,\oraz\, [\Asim  \Asim  \neg p]_\equiv. $$

Most of n-formulae $N_4 p$ could be reduced to some n-formula $N_2 p$ using Proposition~\ref{eq:skr} and Proposition~\ref{simneg:skr}. Remaining n-formulae are divided into three equivalence classes.
\begin{fct}\label{eq:4}
For n-formulae of the length $4$ we have following equivalences: 
\begin{enumerate} 
\item\label{eq:czt} $\Asim  \Asim  \neg \Asim  p \equiv \neg \Asim  \neg \Asim  p \equiv \Asim  \Asim  \neg \neg p \equiv \neg \Asim  \neg \neg p$,
\item\label{eq:pi} $\neg \neg \Asim  \neg p \equiv \neg \Asim  \Asim  \neg p$.
\end{enumerate}
\end{fct}
\begin{fct}
There is only one irreducible n-formula of length $4$ that is not equivalent to any other n-formula $N_4 p$, namely $\neg \neg \Asim  \Asim  p$.
\end{fct}
 Recapitulating, we can distinguish three representatives of irreducible and not equivalent n-formulae of the length $4$ which will denote equivalence classes: $$[\neg \Asim  \Asim  \neg p]_\equiv, [\neg \neg \Asim  \Asim  p]_\equiv \,\oraz \,[\Asim  \Asim  \neg \neg p]_\equiv.$$
\begin{fct}\label{eq:5}
There are only $4$ irreducible n-formulae of the length $5$. These formulae are equivalent:
$\neg \Asim  \Asim  \neg \Asim  p \equiv \neg \Asim  \Asim  \neg \neg p \equiv \neg \neg \Asim  \neg \Asim  p \equiv \neg \neg \Asim  \neg \neg p.$
\end{fct}
The equivalence class of these n-formulae will be denoted by $$[\neg \Asim \Asim \neg \neg p]_\equiv.$$ All remaining n-formulae $N_{5} p$ are reducible to some formulae $N_{3} p$. Equivalences between n-formulae $N_{5} p$ and $N_{3} p$ are based on Proposition~\ref{eq:skr}, Proposition~\ref{simneg:skr} and extensionality. There is only one not obvious case namely $\neg \neg \Asim  \Asim  \neg p \equiv \Asim  \Asim  \neg p$. Indeed, let us suppose that $\not \models \neg \neg \Asim  \Asim  \neg p \rightarrow \Asim  \Asim  \neg p$. Then there exists some model $\mathcal{M}$ in which $\kor \not \Vdash \neg \neg \Asim  \Asim  \neg p \rightarrow \Asim  \Asim  \neg p$. That is there exists some world $u \geq \kor$ such that $u \Vdash \neg \neg \Asim  \Asim  \neg p \oraz u \not \Vdash \Asim  \Asim  \neg p$. From the second condition follows that there exists a world $v \geq u$ such that for every world $w \geq v$ we have $w \not \Vdash p$. This means that the world $w$ is the root of the model $\mathcal{M}$ and the only world of it and $\kor \Vdash p$. On the other hand we have $\kor \Vdash \neg \neg \Asim  \Asim  \neg p$ which implies that $\kor \not \Vdash p$, a contradiction. Let us suppose that $\not \models \Asim \Asim \neg p \rightarrow \neg \neg \Asim \Asim \neg p$. Then again for the root of some model $\mathcal{M}$ we have $\kor \not \Vdash \Asim \Asim \neg p \rightarrow \neg \neg \Asim \Asim \neg p$. Hence there is a world $u \geq \kor$ such that $u \Vdash \Asim \Asim \neg p \oraz u \not \Vdash \neg \neg \Asim \Asim \neg p$. The second part implies that $u = \kor \oraz u \not \Vdash \Asim \Asim \neg p$, a contradiction.
\begin{cor}\label{skr6}
Every n-formula $N_6 p$ is reducible to some n-formula $N_k p$ of length $k \leq 4$.
\end{cor}
\begin{proof}
For every n-formula $N_6 p$ there exists a sequence $N_5$ such that either $N_6 p = \Asim N_5 p$ or $N_6 p = \neg N_5 p$. We consider two cases.

Case 1. The formula $N_5 p$ is reducible to some n-formula $N_3 p$. Then we have either $N_6 p \equiv \Asim N_3 p$ or $N_6 p \equiv \neg N_3 p$, which implies that every n-formula $N_6 p$ is reducible to some n-formula $N_4 p$.

Case 2. The formula $N_5 p$ is irreducible. Then from Fact~\ref{eq:5} follows that $N_5 p \in [\neg \Asim \Asim \neg \neg p]_\equiv.$ If $N_6 p = \Asim N_5 p$ then in particular $N_6 p \equiv \Asim \neg \Asim \Asim \neg \neg p$ and from Proposition~\ref{simneg:skr}.\ref{simneg2} follows that $N_6 p \equiv \Asim \neg p$. If $N_6 p = \neg N_5 p$ then in particular $N_6 p \equiv \neg \neg \neg \Asim \neg \neg p$. From point~\ref{neg:odd} of Proposition~\ref{eq:skr} it follows that $N_6 p \equiv \neg \Asim \neg \neg p$. 
\end{proof}
\begin{thm}\label{skr}
Every n-formula $N_k p$, for $k \geq 6$ is reducible to some n-formula of length $m \leq 5$.
\end{thm}
\begin{proof}
By induction on $k$.
The induction base follows from Corollary~\ref{skr6}.

Assume that $k > 6$. Let $\diamond \in \{\Asim, \neg\}$. Let $N_{k+1} p = \diamond N_{k} p$. From the~induction hypothesis there exists a sequence $N_m$ such that $m \leq 5$ and $N_{k} p \equiv N_{m} p$. Thus we have $N_{k+1} p \equiv \diamond N_{m} p.$ If $m<5$, we get the thesis. Else it is the case of the~induction base.
\end{proof}

\section{Equivalence classes of n-formulae}
As a conclusion from the previous section, especially Facts~\ref{N2:irr}--~\ref{eq:5}, Corollary~\ref{skr6} and Theorem~\ref{skr} we get the exact power of the set $\mathcal{N}/_{\equiv}$. 

\begin{thm}
There are exactly $15$ (up to equivalence) pairwise not equivalent and irreducible n-formulae:
 
\begin{center}
\begin{tabular}{cc}
$p$ & \\
\rule[-2ex]{0pt}{4.5ex} 
$\Asim \Asim p$ & $\Asim p$\\
\rule[-2ex]{0pt}{4.5ex} 
$\Asim \neg p$ & $\neg p$ \\ 
\rule[-2ex]{0pt}{4.5ex} 
$\neg \Asim p$ & $\Asim \Asim \neg p$\\
\rule[-2ex]{0pt}{4.5ex} 
$\neg \neg p$ & $\Asim \neg \neg p$\\
\rule[-2ex]{0pt}{4.5ex} 
$\Asim \Asim \neg \neg p$ & $\neg \Asim \Asim p$\\
\rule[-2ex]{0pt}{4.5ex} 
$\neg \Asim \Asim \neg p$ & $\neg \neg \Asim p$\\
\rule[-2ex]{0pt}{4.5ex} 
$\neg \neg \Asim \Asim p$ & $\neg \Asim \Asim \neg \neg p$\\
\end{tabular} 
\end{center}
\end{thm}

Two following theorems gather all relations between elements of the~set $\mathcal{N}/_{\equiv}$. We present them in the form of implications of n-formulae for chosen representatives of equivalence classes, as this is more readable. By stating that an implication $N_k p \rightarrow N_m p$ is valid we mean that the reverse implication is not valid. 

\begin{thm}\label{LAeven}
Following implications of even n-formulae are valid:
\begin{enumerate}
\item\label{LAe:1a} $\Asim  \neg p \rightarrow p$
\item\label{LAe:2a} $p \rightarrow \neg \neg p$
\item\label{LAe:3a} $\neg \neg p \rightarrow \neg \neg \Asim  \Asim  p$
\item\label{LAe:4a} $\neg \neg \Asim  \Asim  p \rightarrow \neg \Asim  p$
\item\label{LAe:5a} $\neg \Asim  p \rightarrow \Asim  \Asim  \neg \neg p$
\item\label{LAe:1b} $\Asim  \neg p \rightarrow \neg \Asim  \Asim  \neg p$
\item\label{LAe:2b} $\neg \Asim  \Asim  \neg p \rightarrow \neg \neg p$
\item\label{LAe:2c} $p\rightarrow \Asim  \Asim  p$
\item\label{LAe:3c} $\Asim  \Asim  p \rightarrow \neg \neg \Asim  \Asim  p$
\end{enumerate}
\end{thm}
\begin{thm}\label{LAodd}
Following implications of odd n-fomulae are valid:
\begin{enumerate}
\item\label{LAo:1a} $\Asim  \neg \neg p \rightarrow \Asim  p$
\item\label{LAo:2a} $ \Asim  p \rightarrow \neg \neg \Asim  p$
\item\label{LAo:3a} $\neg \neg \Asim  p \rightarrow \neg \Asim  \Asim  p$
\item\label{LAo:4a} $\neg \Asim  \Asim  p \rightarrow \neg p$
\item\label{LAo:5a} $\neg p \rightarrow \Asim  \Asim  \neg p$
\item\label{LAo:1b} $\Asim  \neg \neg p \rightarrow \neg \Asim  \Asim  \neg \neg p$
\item\label{LAo:2b} $\neg \Asim  \Asim  \neg \neg p \rightarrow \neg \neg \Asim  p$
\end{enumerate}
\end{thm}
\begin{proof}
As a proof of both Theorem~\ref{LAeven} and Theorem~\ref{LAodd} we refer the Reader to validity tables in the Appendix. Checking the validity of a formula $N_{k} p \rightarrow N_{m} p$ (and stating that the reverse implication is not valid) is reduced to checking if $\es^{+}(N_{k} p) \varsubsetneq \es^{+}(N_{m} p)$.
\end{proof}
Relations between classes of equivalent formulae are described by Lindenbaum Algebra. In the case of negational monadic fragment of $\ICL$ we cannot create such a structure. So we present a poset $(\mathcal{N}^\ast \cup \{0, \perp, 1\}, \preceq)$, where $\mathcal{N}^\ast $ is the set of chosen representatives of equivalence classes of $\mathcal{N}/_\equiv$. 
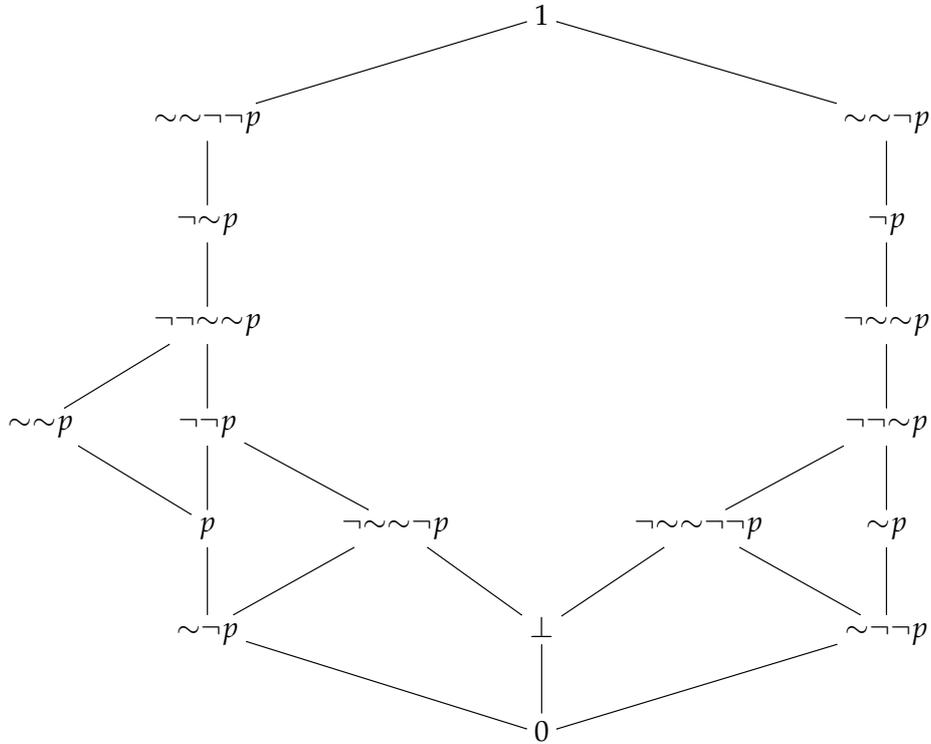
$$
\xymatrix{
 & & & 1 & & \\
 & \Asim  \Asim  \neg \neg p \ar@{-}[urr] & & & & \Asim  \Asim \neg p \ar@{-}[ull] \\
 & \neg \Asim  p \ar@{-}[u] & & & & \neg p \ar@{-}[u] \\
 & \neg \neg \Asim  \Asim  p \ar@{-}[u] & & & & \neg \Asim  \Asim  p \ar@{-}[u]\\
\Asim  \Asim  p \ar@{-}[ur] & \neg \neg p \ar@{-}[u] & & & & \neg \neg \Asim  p \ar@{-}[u]\\
 & p \ar@{-}[ul] \ar@{-}[u] & \neg \Asim  \Asim  \neg p \ar@{-}[ul] & & \neg \Asim  \Asim  \neg \neg p \ar@{-}[ur] & \Asim  p \ar@{-}[u] \\
 & \Asim  \neg p \ar@{-}[u] \ar@{-}[ur] &  & \perp \ar@{-}[ul] \ar@{-}[ur] & & \Asim  \neg \neg p \ar@{-}[ul] \ar@{-}[u] \\
 & & & 0 \ar@{-}[ull] \ar@{-}[u] \ar@{-}[urr] & & \\
}
$$
The addition of constants enables to join the two posets of equivalence classes of even and odd n-formulae.

After looking into properties of negational formulae several questions arose. They consider mainly computational content of n-formulae. Further work would also include investigations of implicational fragment of $\ICL$.

\flushleft
{\em Institute of Mathematics\\ University of Silesia\\ Bankowa 14\\ 40-007 Katowice, Poland}
\par
{\it e-mail:} {\ttfamily aglenszczyk@us.edu.pl}
\newpage

\section*{Appendix}
We give complete validity tables for n-formulae $N_{k} p$ up to the length of $5$. 
\begin{center}
\medskip
\begin{tabular}{|c|c|c|c|c|c|c|c|c|c|}
\hline \rule[-2ex]{0pt}{5ex}
 & $\mz$ & $\mj$ & $\mzz$ & $\mzj$ & $\mjj$ & $\vau$ & $i(\mz)$ & $i(\mj)$ & $i(\mz, \mj)$\\ 
\hline\rule[-2ex]{0pt}{5ex}  
$p$ 
& - & + & - & - & + & - & - & + & +
\\ 
\hline\rule[-2ex]{0pt}{5ex}  
$\Asim p$ 
& + & - & + & - & - & - & + & - & - 
\\ 
\hline\rule[-2ex]{0pt}{5ex}  
$\neg p$ 
&  + & -  & +  & + & - & + & + & + & + 
\\ 
\hline\rule[-2ex]{0pt}{5ex}  
$\Asim\Asim p$ 
& - & + & - & + & + & - & - & + & - 
\\ 
\hline\rule[-2ex]{0pt}{5ex}  
$\Asim\neg p$ 
& - & + & - & - & - & - & - & - & - 
\\ 
\hline\rule[-2ex]{0pt}{5ex}  
$\neg\Asim p$ 
& - & + & - & + & + & + & + & + & + 
\\ 
\hline\rule[-2ex]{0pt}{5ex}  
$\neg\neg p$ 
& - & + & - & - & + & - & + & + & + 
\\ 
\hline\rule[-2ex]{0pt}{5ex}  
$\Asim\Asim\Asim p$ 
& + & - & + & - & - & - & + & - & - 
\\ 
\hline\rule[-2ex]{0pt}{5ex}  
$\Asim\Asim\neg p$ 
& + & - & + & + & + & + & + & + & + 
\\ 
\hline\rule[-2ex]{0pt}{5ex}  
$\Asim\neg\Asim p$ 
& + & - & - & - & - & - & - & - & -
\\ 
\hline\rule[-2ex]{0pt}{5ex}  
$\Asim\neg\neg p$ 
& + & - & - & - & - & - & - & - & - 
\\ 
\hline\rule[-2ex]{0pt}{5ex}  
$\neg\Asim\Asim p$ 
& + & - & + & - & - & + & + & + & +
\\
\hline\rule[-2ex]{0pt}{5ex}  
$\neg\Asim\neg p$ 
& + & -  & + & + & + & + & + & + & +
\\ 
\hline\rule[-2ex]{0pt}{5ex}  
$\neg\neg\Asim p$ 
& + & - & + & - & - & - & + & + & +
\\ 
\hline\rule[-2ex]{0pt}{5ex}  
$\neg\neg\neg p$ 
& + & - & + & + & - & + & + & + & +
\\ 
\hline\rule[-2ex]{0pt}{5ex}  
$\Asim\Asim\Asim\Asim p$ 
& - & + & - & + & + & - & - & + & - 
\\ 
\hline\rule[-2ex]{0pt}{5ex}  
$\Asim\Asim\Asim\neg p$ 
& - & + & - & - & - & - & - & - & -
\\ 
\hline\rule[-2ex]{0pt}{5ex}  
$\Asim\Asim\neg\Asim p$ 
& - & + & + & + & + & + & + & + & +
\\ 
\hline\rule[-2ex]{0pt}{5ex}  
$\Asim\Asim\neg\neg p$ 
& - & + & + & + & + & + & + & + & +
\\
\hline 
\end{tabular}

\medskip

\begin{tabular}{|c|c|c|c|c|c|c|c|c|c|}
\hline\rule[-2ex]{0pt}{5ex} 
 & $\mz$ & $\mj$ & $\mzz$ & $\mzj$ & $\mjj$ & $\vau$ & $i(\mz)$ & $i(\mj)$ & $i(\mz, \mj)$\\  
\hline\rule[-2ex]{0pt}{5ex}  
$\Asim\neg\Asim\Asim p$ 
& - & + & - & - & - & - & - & - & -
\\
\hline\rule[-2ex]{0pt}{5ex}  
$\Asim\neg\Asim\neg p$ 
& - & + & - & - & - & - & - & - & -
\\ 
\hline\rule[-2ex]{0pt}{5ex}  
$\Asim\neg\neg\Asim p$ 
& - & + & - & - & - & - & - & - & -
\\ 
\hline\rule[-2ex]{0pt}{5ex}  
$\Asim\neg\neg\neg p$ 
& - & + & - & - & - & - & - & - & -
\\
\hline\rule[-2ex]{0pt}{5ex}  
$\neg\Asim\Asim\Asim p$ 
& - & + & - & + & + & + & + & + & +
\\ 
\hline\rule[-2ex]{0pt}{5ex}  
$\neg\Asim\Asim\neg p$ 
& - & + & - & - & - & - & + & + & +
\\ 
\hline\rule[-2ex]{0pt}{5ex}  
$\neg\Asim\neg\Asim p$ 
& - & + & + & + & + & + & + & + & +
\\ 
\hline\rule[-2ex]{0pt}{5ex}  
$\neg\Asim\neg\neg p$ 
& - & + & + & + & + & + & + & + & +
\\ 
\hline\rule[-2ex]{0pt}{5ex}  
$\neg\neg\Asim\Asim p$ 
& - & + & - & + & + & - & + & + & +
\\
\hline\rule[-2ex]{0pt}{5ex}  
$\neg\neg\Asim\neg p$ 
& - & + & - & - & - & - & + & + & +
\\ 
\hline\rule[-2ex]{0pt}{5ex}  
$\neg\neg\neg\Asim p$ 
& - & + & - & + & + & + & + & + & +
\\ 
\hline\rule[-2ex]{0pt}{5ex}  
$\neg\neg\neg\neg p$ 
& - & + & - & - & + & - & + & + & +
\\ 
\hline\rule[-2ex]{0pt}{5ex}  
$\Asim \Asim \Asim \Asim \Asim p$
& + & - & + & - & - & - & + & - & - 
\\ 
\hline\rule[-2ex]{0pt}{5ex} 
$\Asim \Asim \Asim \Asim \neg p$
& + & - & + & + & + & + & + & + & +
\\ 
\hline\rule[-2ex]{0pt}{5ex} 
$\Asim \Asim \Asim \neg \Asim p$
& + & - & - & - & - & - & - & - & -
\\ 
\hline\rule[-2ex]{0pt}{5ex}  
$\Asim \Asim \Asim \neg \neg p$
& + & - & - & - & - & - & - & - & -
\\ 
\hline\rule[-2ex]{0pt}{5ex}  
$\Asim \Asim \neg \Asim \Asim p$
& + & - & + & + & + & + & + & + & +
\\ 
\hline\rule[-2ex]{0pt}{5ex}  
$\Asim \Asim \neg \Asim \neg p$
& + & - & + & + & + & + & + & + & +
\\ 
\hline\rule[-2ex]{0pt}{5ex} 
$\Asim \Asim \neg \neg \Asim p$
& + & - & + & + & + & + & + & + & +
\\ 
\hline\rule[-2ex]{0pt}{5ex} 
$\Asim \Asim \neg \neg \neg p$
& + & - & + & + & + & + & + & + & +
\\ 
\hline\rule[-2ex]{0pt}{5ex} 
$\Asim \neg \Asim \Asim \Asim p$
& + & - & - & - & - & - & - & - & -
\\ 
\hline\rule[-2ex]{0pt}{5ex} 
$\Asim \neg \Asim \Asim \neg p$
& + & - & - & - & - & - & - & - & -
\\ 
\hline 
\end{tabular} 

\medskip

\begin{tabular}{|c|c|c|c|c|c|c|c|c|c|}
\hline\rule[-2ex]{0pt}{5ex} 
 & $\mz$ & $\mj$ & $\mzz$ & $\mzj$ & $\mjj$ & $\vau$ & $i(\mz)$ & $i(\mj)$ & $i(\mz, \mj)$\\  
\hline\rule[-2ex]{0pt}{5ex} 
$\Asim \neg \Asim \neg \Asim p$
& + & - & - & - & - & - & - & - & -
\\ 
\hline\rule[-2ex]{0pt}{5ex} 
$\Asim \neg \Asim \neg \neg p$
& + & - & - & - & - & - & - & - & -
\\ 
\hline\rule[-2ex]{0pt}{5ex} 
$\Asim \neg \neg \Asim \Asim p$
& + & - & - & - & - & - & - & - & -
\\ 
\hline\rule[-2ex]{0pt}{5ex} 
$\Asim \neg \neg \Asim \neg p$
& + & - & - & - & - & - & - & - & -
\\ 
\hline\rule[-2ex]{0pt}{5ex} 
$\Asim \neg \neg \neg \Asim p$
& + & - & - & - & - & - & - & - & -
\\ 
\hline\rule[-2ex]{0pt}{5ex}
$\Asim \neg \neg \neg \neg p$
& + & - & - & - & - & - & - & - & -
\\ 
\hline\rule[-2ex]{0pt}{5ex}  
$\neg \Asim \Asim \Asim \Asim p$
& + & - & + & - & - & + & + & + & +
\\
\hline\rule[-2ex]{0pt}{5ex} 
$\neg \Asim \Asim \Asim \neg p$
& + & -  & + & + & + & + & + & + & +
\\ 
\hline\rule[-2ex]{0pt}{5ex} 
$\neg \Asim \Asim \neg \Asim p$
& + & - & - & - & - & - & + & + & +
\\ 
\hline\rule[-2ex]{0pt}{5ex} 
$\neg \Asim \Asim \neg \neg p$
& + & - & - & - & - & - & + & + & +
\\ 
\hline\rule[-2ex]{0pt}{5ex} 
$\neg \Asim \neg \Asim \Asim p$
& + & -  & + & + & + & + & + & + & +
\\ 
\hline\rule[-2ex]{0pt}{5ex} 
$\neg \Asim \neg \Asim \neg p$
& + & -  & + & + & + & + & + & + & +
\\ 
\hline\rule[-2ex]{0pt}{5ex}  
$\neg \Asim \neg \neg \Asim p$
& + & -  & + & + & + & + & + & + & +
\\ 
\hline\rule[-2ex]{0pt}{5ex} 
$\neg \Asim \neg \neg \neg p$
& + & -  & + & + & + & + & + & + & +
\\ 
\hline\rule[-2ex]{0pt}{5ex}  
$\neg \neg \Asim \Asim \Asim p$
& + & - & + & - & - & - & + & + & +
\\ 
\hline\rule[-2ex]{0pt}{5ex}  
$\neg \neg \Asim \Asim \neg p$
& + & -  & + & + & + & + & + & + & +
\\ 
\hline\rule[-2ex]{0pt}{5ex}
$\neg \neg \Asim \neg \Asim p$
& + & - & - & - & - & - & + & + & +
\\ 
\hline\rule[-2ex]{0pt}{5ex}
$\neg \neg \Asim \neg \neg p$
& + & - & - & - & - & - & + & + & +
\\ 
\hline\rule[-2ex]{0pt}{5ex}
$\neg \neg \neg \Asim \Asim p$
& + & - & + & - & - & + & + & + & +
\\
\hline\rule[-2ex]{0pt}{5ex}  
$\neg \neg \neg \Asim \neg p$
& + & -  & + & + & + & + & + & + & +
\\ 
\hline\rule[-2ex]{0pt}{5ex}  
$\neg \neg \neg \neg \Asim p$
& + & - & + & - & - & - & + & + & +
\\ 
\hline\rule[-2ex]{0pt}{5ex}  
$\neg \neg \neg \neg \neg p$
& - & + & - & - & + & - & + & + & +
\\ 
\hline 
\end{tabular} 
\end{center}

\end{document}